\documentclass[a4paper,12pt]{amsart}
\usepackage[pdftex]{hyperref} 
\usepackage{appendix}

\usepackage{tikz}
\usepackage{amsaddr}
\usepackage{amsthm, amssymb, amsmath}
\usepackage{color}
\usepackage{enumerate}
\usepackage{thmtools,thm-restate}
\usepackage{mathtools}                                                                                                                                                                                                     

\usetikzlibrary{shapes.geometric, arrows}
\tikzstyle{startstop} = [rectangle,rounded corners, align=center, draw=black]
\tikzstyle{io} = [rectangle,rounded corners, align=center, draw=black]
\tikzstyle{process} = [rectangle,rounded corners, align=center, draw=black]
\tikzstyle{decision} = [rectangle,rounded corners, align=center, draw=black]

\newtheorem{thm}{Theorem}[section]
\newtheorem*{thm*}{Theorem}
\newtheorem{cor}[thm]{Corollary}

\newtheorem*{claim*}{Claim}
\newtheorem{lem}[thm]{Lemma}
\newtheorem{prop}[thm]{Proposition}

\theoremstyle{definition}
\newtheorem{defn}[thm]{Definition}

\newtheorem{nt}[thm]{Notation}

\newtheorem{rem}[thm]{Remark}
\newtheorem{ex}[thm]{Example}
\newtheorem*{rem*}{Remark}

\DeclareMathOperator{\Ad}{Ad}

\tikzset{every picture/.style={line width=0.75pt}} 

\def \Z{{\mathbb Z}}

\def\Aut{\operatorname{Aut}}
\def\Out{\operatorname{Out}}
\def\Inn{\operatorname{Inn}}

\def\rank{\operatorname{rank}}

\def\semidirect{\rtimes}

\newcommand{\fix}{\mbox{\rm Fix } \phi}

\newcommand{\Chi}{{\rm X}} 

\title{The conjugacy problem for $\Out(F_3)$.}
\author{Fran\c{c}ois Dahmani}
\address{Univ. Grenoble Alpes, CNRS, IF, 38000 Grenoble, France, and \\ IRL-CRM CNRS, Universit\'e de Montr\'eal, Montr\'eal, Canada.}
\email{francois.dahmani@univ-grenoble-alpes.fr}

\author{Stefano Francaviglia}
\address{Dipartimento di Matematica of the University of
	Bologna}
\email{stefano.francaviglia@unibo.it}
\author{Armando Martino}
\address{Mathematical Sciences, University of Southampton }
\email{A.Martino@soton.ac.uk}

\author{Nicholas Touikan}
\address{Department of Mathematics \& Statistics\\ University of New Brunswick (Fredericton)}
\email{nicholas.touikan@unb.ca}

\begin{document}

\maketitle
\begin{abstract} We present a solution to the Conjugacy Problem in the group of outer-automorphisms of $F_3$, a free group of rank 3.  We distinguish according to several computable invariants, such as irreducibility, subgroups of polynomial growth, and subgroups carrying the attracting lamination.  We establish, by considerations on train tracks,  that the conjugacy problem is decidable for the outer-automorphisms  of $F_3$ that preserve a given rank 2 free factor. Then we establish, by consideration on mapping tori, that it is decidable for outer-automorphisms  of $F_3$ whose maximal polynomial growth subgroups are cyclic. This covers all the cases left by the state of the art. 
\end{abstract}


\pagebreak
\tableofcontents


\section{Introduction}

In this paper we solve Dehn's conjugacy problem in $\Out(F_3)$, the group of outer automorphisms of the free group $F_3$ of rank 3: 

\begin{thm}\label{thm:main}
    There exists an algorithm which takes two automorphisms $\phi,\psi\ \in Aut(F_3)$ and correctly outputs \textbf{yes} or \textbf{no} whether their outer classes are conjugate in $\Out(F_3)$.
\end{thm}

The conjugacy problem has already been solved within certain classes  of outer automorphisms of free groups:  the atoroidal fully irreducible ones   by Sela \cite{Sela}, all the irreducible ones by Los \cite{Los},  all the atoroidal ones  by  \cite{Dahmani2016},   the linearly growing ones  by Kristi\'c, Lustig, and Vogtmann \cite{KrsticLustigVogtmann},  the unipotent polynomially growing ones  by Feighn and Handel \cite{Feighn_Handel_Conjugacy}.

In the case of a free group of rank 2, $\Out(F_2) \simeq \mathrm{GL}(2,\mathbb Z)$ is well known and virtually free, and there is an       
algorithm to solve its conjugacy problem.

In the case of a free group of rank 3, $\Out(F_3)$ is more complicated, and our conjugacy decision algorithm operates by classifying the pair of input automorphisms by means of invariants, in subclasses, in which  specific methods can be applied.

Let us mention key conjugacy invariants  that are relevant in these classifications,  and are computable. They give a first frame of reference.

\begin{itemize}
    \item Irreducibility (whether there is an invariant conjugacy class of free factor system, or not), \cite{Kapovich_algo}, \cite{Francaviglia2022};
    \item exponential growth (whether there is a conjugacy class whose iterated images grow exponentially fast, or not), \cite{Bestvina2005}; 
    \item  rank of the maximal polynomially growing subgroups (the maximal subgroups whose elements have iterated images that grow at most polynomially in conjugacy length), and  polynomial degree of growth in these subgroups (Proposition \ref{prop:alggrowth});
    \item arrangement of these subgroups (their orbit under the automorphism group of the ambiant free group) (Gersten's algorithm);
    \item ranks of the free factors carrying the so-called attracting lamination  (Proposition \ref{prop;tt_for_lamination_ffs}). 
\end{itemize}

\begin{ex}\label{ex;quadratic} The map $a\mapsto a$, $b\mapsto ba$, $c\mapsto cb$ extends to the free group of basis $\{a,b,c\}$ as an automorphism that is reducible ($\langle a, b\rangle $ is an invariant free factor of rank $2$),  polynomially growing of degree $2$ (the growth of $c$ is quadratic). 
\end{ex}

\begin{ex}\label{ex;toral}
A pseudo-Anosov mapping class on a twice punctured torus  gives an automorphism that is reducible (each puncture corresponds to a free factor of rank one that is preserved), of exponential growth, with two conjugacy classes of maximal polynomially growing subgroups (corresponding to the punctures), both cyclic,  on which the automorphism has polynomial degree 0. They are both rank 1 free factors, but not in a same free factor system.  The attracting lamination is supported by the entire group.     
\end{ex}
\begin{ex}\label{ex;TwedgeC}
    Take $T$ a torus with one hole $\partial T$, with a base point on its  boundary, and a circle $C$ with a base point, and take the wedge of these spaces, identifying the two base points: its fundamental group is $F_3$.   Consider a pseudo-Anosov mapping class  on $T$, fixing  the boundary pointwise, and a map that sends $C$ on the concatenation $C\cdot \ell$ for a chosen loop $\ell$ in $T$.  The defined map on $T\wedge C$  induces an automorphism of $F_3$ that is reducible, of exponential growth, with an invariant free factor of rank $2$ (the group of $T$). The cyclic group of the boundary $\partial T$ of $T$ is  polynomially growing of degree $0$. Depending on the choice of $\ell$, it might or might not be a maximal polynomially growing subgroup: if $\ell$ is a power of the loop $\partial T$, the maximal polynomially growing subgroup containing $\pi_1(\partial T)$ is actually of rank 2, on which the automorphism  has polynomial degree $0$ or $1$, and it  is not a free factor of $F_3$ but rather a factor of a decomposition of $F_3$ as some amalgamated free product \[\pi_1(T) *_{\pi_1(\partial T)}  \pi_1 (C \wedge \partial T).\]   The attracting lamination is  carried by the rank 2 free factor $\pi_1(T)$. 
\end{ex}

 It is nevertheless not sufficient to collect these invariants to have characterised the conjugacy class of an outer-automorphism. For instance, knowing that the automorphisms are irreducible, with pure exponential growth (i.e. the only polynomially growing subgroup is the trivial one) still requires stamina to decide the conjugation between such outer-automorphisms. Actually, all the solved cases listed in the beginning of this introduction all correspond to some particular situation in this frame of reference. 
 
 In our study of $\Out(F_3)$, we will observe that, given the state of the art,  only three configurations in this frame of reference need to be treated. We will cast two of them in the point of view of automorphisms preserving a specific rank 2 free factor. We will cast the last one in the point of view of toral relatively hyperbolic mapping tori. The first configurations is that of the polynomial growth of degree $2$ on the group $F_3$ (well illustrated by Example \ref{ex;quadratic}). The second configuration is the case of exponential growth in which a rank 2 free factor is invariant, and "attracts" all the growth.  In that case (that is well illustrated by our example \ref{ex;TwedgeC}) there is only one class of maximal polynomially growing subgroup, it is not a free factor, it is  either  cyclic (generated by the commutator of a basis in the invariant rank 2 free factor),  or of rank 2  (attached to this free factor). The last case (illustrated by Example \ref{ex;toral})  is when the growth is exponential, with a rank 1 free factor preserved, and the maximal polynomial growth subgroups are of rank 1 (there might  
  be two of them up to conjugacy).   

Our algorithm is shown in Figure \ref{fig:THE_algo} and its correctness is derived from the correctness of every classification step as well as every terminal step. There are several ways to separate cases 
toward the use of Theorem \ref{thm:rank_2_reducible} or the use of Theorem \ref{thm;solve-atrh} or of methods in section \ref{sec:expo_ii},
since they overlap, but considering complexity it seems reasonable to leave the later for the smaller number of cases.

 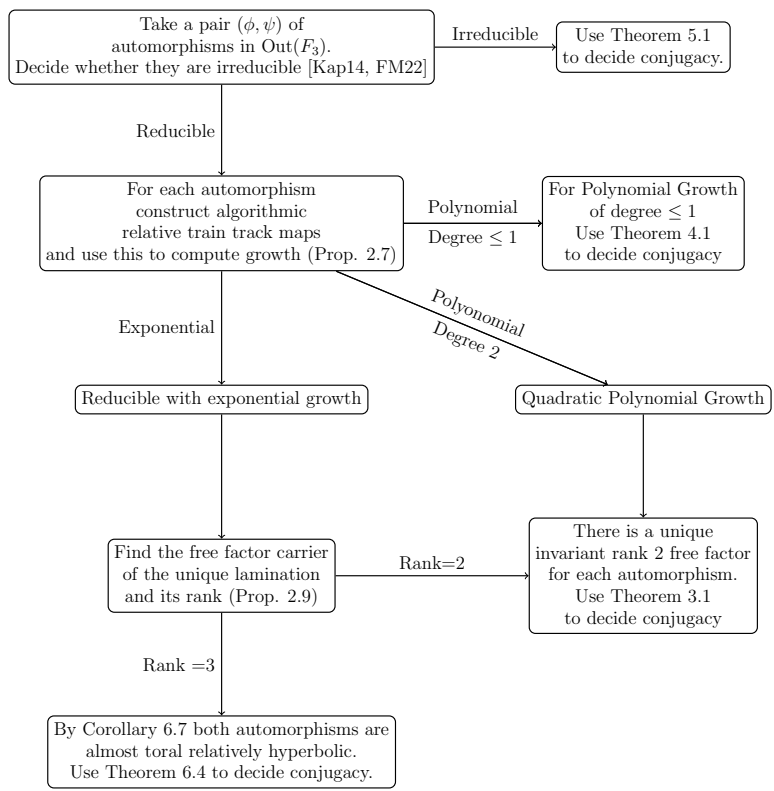
\begin{figure}[htb]
	\centering
\resizebox{.8\textwidth}{!}
{ 	
\begin{tikzpicture}[node distance=4cm]
\node (input) [io] {Take a pair $(\phi,\psi)$ of  \\automorphisms in $\Out(F_3)$. \\ Decide whether they are irreducible \cite{Kapovich_algo, Francaviglia2022}};
\node (irred)	[io, right of=input, xshift=5.5cm] {Use Theorem \ref{thm:CP-iwip} \\ to decide conjugacy.};
\draw[->] (input) -- node[above]{Irreducible}(irred); 	
\node (reducible) [io, below of=input]{For each automorphism \\ construct algorithmic \\ relative train track maps \\  and use this to compute growth (Prop. \ref{prop:alggrowth})};
\draw[->] (input)  -- node[left]{Reducible}(reducible);
\node (exponential) [io, below of=reducible]{Reducible with exponential growth}; 
\draw[->] (reducible) -- node[left]{Exponential} (exponential);
\node (poly) [io, right of=reducible, xshift=5.5cm]{For Polynomial Growth \\ of degree $\leq 1$ \\ Use Theorem \ref{thm:CP_subquadratic}\\ to decide conjugacy};
\draw[->](reducible) -- node[above]{Polynomial}(poly)      ; 
\draw[->](reducible) -- node[below]{Degree $\leq 1$}(poly)      ; 
\node (deg2) [io, below of=poly]{Quadratic Polynomial Growth}  ; 
\draw[->] (reducible) -- node[sloped, above]{Polyonomial} (deg2); 
\draw[->] (reducible) -- node[sloped, below]{Degree 2} (deg2); 
\node (lamination) [io, below of=exponential]{Find the free factor carrier \\ of the unique lamination \\ and its rank (Prop. \ref{prop;tt_for_lamination_ffs})}; 
\draw[->] (exponential) -- (lamination);
\node (rank2) [io, below of=deg2]{There is a unique \\invariant rank 2 free factor  \\ for each automorphism.  \\ Use Theorem \ref{thm:rank_2_reducible} \\ to decide conjugacy};
\draw[->] (deg2) -- (rank2); 
\draw[->] (lamination) -- node[above]{Rank=2}(rank2); 
\node (filling) [io, below of=lamination]{By Corollary \ref{cor;no_iff2_atrh} both automorphisms are \\almost toral relatively hyperbolic. \\ Use Theorem \ref{thm;solve-atrh} to decide conjugacy. }; 
\draw[->] (lamination) -- node[left]{Rank =3}(filling);
\end{tikzpicture}}
     \caption{Our algorithm to solve the conjugacy problem in $\Out(F_3)$. If at any time the  given automorphisms $\phi$ and $\psi$ follow different arrows, they are not conjugate.}  \label{fig:THE_algo}
\end{figure}

    In Section \ref{sec:rank2_reduction} we show how to solve the conjugacy problem in $\Out(F_3)$ among all outer automorphisms with a given invariant rank 2 free factor. This result, in a sense, is a type of induction step from $\Out(F_2)$ to $\Out(F_3)$, but as will be explained in Remark \ref{rem:no_induction}, it will not generalize similarly to higher ranks. This case is critically important as we will show that most problematic situations (which actually fall in the case  where the given automorphism induces a polynomial growth automorphism on a non-cyclic subgroup) lead to an invariant free factor of rank 2.

    In Section \ref{sec:quadratic} we show that quadratically growing outer automorphisms have an invariant rank 2 free factor. In Section \ref{sec:expo_i} we use train track methods to study the types of attracting laminations that arise. One possibility gives rise (again) to an an invariant rank 2 free factor. In Section \ref{sec:expo_ii} we show that the other possibility implies that the outer automorphisms are so-called almost toral relatively hyperbolic and the conjugacy problem for this class of outer automorphisms is established by  the main result of \cite{DahmaniTouikan2021} in its particular case where  the so-called peripheral subgroups are $\mathbb{Z}\times \mathbb{Z}$  or  $\mathbb Z \rtimes \mathbb Z$, the fundamental group  of the Klein bottle.

          \subsection*{Acknowledgements} We would like to thank Vincent Guirardel, Mark Feighn and Michael Handel for valuable discussions. 

            F.D. is supported by ANR-22-CE40-0004 GoFR. S.F. is supported by INdAM group GNSAGA and the European Union - NextGenerationEU under the National Recovery and Resilience Plan (PNRR) – Mission 4 Education and Research – Component 2 From Research to Business - Investment 1.1, Notice Prin 2022 issued with Decree No. 104 on 2/2/2022, with the title `Geometry and Topology of Manifolds,' proposal code 2022NMPLT8 - CUP J53D23003820001. N.T. is supported by an NSERC Discovery Grant.

          \section{Computability of a few invariants}
          
\subsection*{Notations}
          
    We begin by some general notations, and definitions.

\begin{nt}\label{treedef}
~	   
    \begin{itemize}
		\item In a group $G$, the conjugation of $g$ by $h$ is  $g^h:=h^{-1} g h$.
	    \item $\Aut(G)$ acts on the right: for $\phi \in \Aut(G)$, $g \in G$, the image of $g$ under $\phi$ is $g \phi$.	
		\item $\Ad(g)$ denotes the automorphism of $G$ that is conjugation by $g$: $x\Ad(g)=x^g$.  Hence $\Ad(g) \Ad(h) = \Ad(gh)$. 
        \item $\Out(G) = \Aut(G)/ \Inn (G)$, for $\Inn (G) = \{ \Ad(g), g\in G\}$.
		\item   $F_3= \langle a,b,c \rangle$ denotes a free group of rank $3$. 
        \item We will prefer the notation $K$ for free factors, and $H$ for subgroups. 
	       \item  We write outer-automorphisms with capital greek letters, and automorphisms with small greek letter: if $\Phi$ is an outer-automorphism, $\phi\in \Phi$ is an automorphism in its class.   In this convention, (according to context) $\Chi$ reads {\it Chi}, and $\chi \in \Chi$ is an automorphism.
    \end{itemize}
\end{nt}

    For $\Phi \in \Out(F)$, we say a subgroup $H\leq F$ is \emph{$\Phi$-invariant} if for any $\phi \in \Phi$ we have that $H$ is conjugate to $H\phi$. 

    A \emph{free factor system} of  $F$ is a set of conjugacy classes of subgroups of $F$,  $\{[K_1], \dots, [K_m]\}$ such that there exists a free subgroup $F_r<F$ for which  $F= K_1 * \dots *K_m * F_r$.   It is proper if the $K_i$ are neither $\{1\}$ nor $F$.    It is $\Phi$ invariant if all $K_i$ (not necessarily $F_r$) are $\Phi$ invariants.

    A free factor system $\{[Y_1], \dots, [Y_s]\}$ is \emph{smaller} than $\{[K_1], \dots, [K_m]\}$ if for each $i$ there exists $j$ such that  $Y_i$ has conjugate inside $K_j$. 

    \subsection*{Irreducibility}\label{sec;decide_irred}

    A first conjugacy invariant, and decidable, property is the irreducibility of outer-automorphisms.

    \begin{defn}
     An outer-automorphism $\Phi$ is called \emph{reducible} if it admits an invariant proper, non-trivial, free factor system - see \cite{Bestvina2000}. Otherwise it is said to be \emph{irreducible}. 

     An automorphism is  \emph{fully irreducible} if  every positive power is irreducible. 
    \end{defn}

   \begin{rem}
   An automorphism is fully irreducible if and only if the only \emph{periodic} free factors preserved up to conjugacy are given by either the trivial subgroup or the whole group. Every fully irreducible automorphism is irreducible, but not conversely. 

   An example of an irreducible automorphism which is not fully irreducible is as follows: take a surface with $p>1$ punctures. Consider a pseudo-Anosov map on the surface which cyclically permutes the punctures. Then the outer automorphism induced on the fundamental group is irreducible but not fully irreducible. In fact, any example of exponential growth arises in this way. 

   Fully irreducible is often referred to as `iwip' - irreducible with irreducible powers - in the literature. 
   \end{rem}    

    There is an algorithm to detect whether $\Phi \in \Out(F_n)$ is fully irreducible given in \cite{Kapovich_algo}, and whether it is irreducible in \cite{Francaviglia2022}. A more recent algorithm given in \cite{kapovich_detecting_2019} decides whether $\Phi$ is fully irreducible in polynomial time.


    \subsection*{Growth}
    
    A second set of conjugacy invariant, and decidable, properties is related to growth.

\begin{defn}
    Given $\phi\in\Phi \in \Out(F_n)$ and $g \in F_n$ and a fixed basis $X$ of $F_n$ we define the growth rate of $g$ as follows:

    \begin{itemize}
        \item The element $g$ has polynomial growth of degree at most $d$ under $\Phi$ if $\|g\phi^n\|_X = O(n^d)$. 
        \item The element $g$ has exponential growth if there exists $\lambda >1$ for which $ \lambda^n = O(\|g\phi^n\|_X )$. 
\footnote{Note that (for non-negative functions) $f(n) = O(g(n))$ means $f(n) \leq M g(n)$ for some constant $M$. The reason for the asymmetric definition of growth is that the idea to bound polynomial growth from above and exponential growth from below. It is easy to see that there is always an exponential upper bound.}

    \end{itemize}

\end{defn}

 This growth rate is independent of the choice of $X$ and $\phi \in \Phi$. 
Moreover, by \cite{Bestvina1992},  the growth rate of every $g\in F_n$  is  either at least exponential or at most polynomial.

\begin{defn}
An outer-automorphism  $\Phi$ of $F_n$ is said to have
\begin{itemize}
    \item \emph{polynomial growth} if all elements $g\in F_n$ have polynomial growth rate,
    \item \emph{exponential growth} if \emph{some} element $g \in F_n$ has exponential growth rate. 
\end{itemize}
\end{defn}
 In the case of a polynomial growth outer-automorphism of $F_n$, there exists $d\geq 0$ for which all elements of $F$ have  polynomial growth of degree at most $d$, and moreover, the smallest such $d$ satisfies  $d \leq n-1$,  see \cite{Bestvina2005}, \cite{Levitt2009}.

\begin{defn}
    A polynomial growth outer-automorphism of $F_n$ is \emph{unipotent} if it induces an automorphism of the abelianisation $\mathbb{Z}^n$ that is represented by a unipotent matrix.   
\end{defn}

Such a matrix of the abelianisation of an automorphism of $F_n$ of polynomial growth only has eigenvalues of modulus $1$, and in in $GL(n, \mathbb{Z})$, so 
its $|GL(n, \mathbb{Z}/3)|$-th power is unipotent.

    The following is known to specialists, and important to us.  We explain a proof, using the theory of CT maps \cite{Feighn_Handel_2018}, that are certain homotopy equivalences on graphs,  representing outer-automorphisms on their fundamental groups; the interested reader is referred to this reference.

    \begin{prop}[\cite{Feighn2011, Feighn_Handel_2018}] \label{prop:alggrowth}
    Given an automorphism, $\Phi  \in \Out(F_n)$, there is an algorithm to decide: 
    \begin{itemize}
        \item If $\Phi$ has polynomial or exponential growth and, 
        \item If $\Phi$ has polynomial growth, decide the degree of polynomial growth. 
    \end{itemize}
\end{prop}
\begin{proof}
By \cite{Feighn_Handel_2018}, there is an algorithm to construct a CT map (which is, in particular, a relative train track map) representing a rotationless power of $\Phi$. Since the properties needed here are invariant under taking positive powers, we may as well assume that $\Phi$ itself is so-called rotationless (the power needed can be determined in advance). In particular, a rotationless automorphism of polynomial growth would be unipotent, see \cite[Lemma 4.2.2]{Feighn2011}.

So take a CT map, $f: G \to G$ representing $\Phi$. 

It then follows immediately that if $f$ admits an exponential stratum then $\Phi$ will have exponential growth. So we can algorithmically distinguish between exponential and polynomial growth. 

So if $\Phi$ has polynomial growth, then all strata will be NEG (non-exponentially growing); note that by \cite{Feighn2011}, Theorem 2.19, there will be no zero strata in the absence of exponential strata. Then Lemma 4.2 of \cite{Feighn2011} shows that each such NEG stratum consists of a single edge, $E$, such that $f(E) = E \cdot u$, where the dot denotes a \emph{splitting}. In particular, the growth of $f^k(E)$ is polynomial of one degree higher than that of $u$. Hence, by induction, we can determine the growth of every edge. 

This almost determines the growth of the automorphism, $\Phi$, in the sense that $\Phi$ has polynomial growth of degree $d >1$ if and only if some edge grows polynomially with degree $d$. However, it is possible to represent the identity map as a CT map where some edge grows linearly and this is the only exception. See \cite[Lemma 2.16]{macura_detour_2002} and \cite[Lemma 2.3]{Andrew2022a}. 

However, since we have already passed to a rotationless power (in particular it will be UPG), our automorphism will have growth of degree 0 if and only if it represents the identity and this is immediately checked. 

\end{proof}

For an outer-automorphism of $F_n$ of exponential growth, there exists a canonical finite collection of conjugacy classes of finitely generated subgroups of $F_n$, such that an element $g\in F_n$  has polynomial growth rate if and only if it is conjugate in one of these subgroups \cite{Levitt2009}. Those are the maximal polynomially growing subgroups. 

Given $H$ such a maximal subgroup, and $\phi \in \Phi$, there is an integer $m> 0$, and $g\in F_n$ for which $H( \phi^m \Ad (g))=H$. The automorphism $\phi^m \Ad (g)$ of $H$ is then a polynomial growth automorphism. We again refer to \cite{Levitt2009}. Through standard arguments of quasi-isometry, its polynomial degree of growth rate does not depend on $\phi \in \Phi$, $m$ and $g$ satisfying the above. We thus call it the degree of growth of $\Phi$ on $H$. 

\begin{prop}\label{prop2.8}
    There is an algorithm that, given $\Phi \in \Out(F_n)$ of exponential growth, computes a basis for each conjugacy-representative of  maximal polynomially growing subgroup of $F_n$, and the degree of growth of $\Phi$ on it.
\end{prop}

Proposition~\ref{prop2.8} follows from the relative hyperbolicity of $F_n \rtimes \mathbb{Z}$ given by $\Phi$ \cite{ghosh_relative_2023, DamaniLi2020}, and the computability of its peripheral subgroups \cite{DahmaniGuirardel2013}. This detects the maximal polynomially growing subgroups of $F_n$. Then, Proposition \ref{prop:alggrowth} determines the induced growth on each of them. 

\subsection*{Lamination carriers are computable}
A more involved conjugacy invariant for outer-automorphisms of free groups that have exponential growth is that of the carrier of the attracting laminations \cite[Section 3]{Bestvina2000}. This is a specific conjugacy class of free factor systems of $F_n$, that in some sense, attracts all the  exponential growth of the automorphism. 

More concretely, let $\partial F_n$, denote Gromov boundary of $F_n$, which for a free group is the same as the set of ends. Let $\widetilde{\mathcal{B}} = (\partial F_n \times \partial F_n - \Delta)/\Z_2$; where $\Delta$ is the diagonal subset of $\partial F_n \times \partial F_n$, and the $\Z_2$ action is via interchanging the coordinates. That is, $\widetilde{\mathcal{B}}$ can be thought of as the set of unordered pairs of distinct points on the boundary of $F_n$; morally, this is the set of lines in $F_n$. The action of $F_n$ on its boundary extends to an action on $\widetilde{\mathcal{B}}$. 

We then let $\mathcal{B} = \widetilde{\mathcal{B}}/F_n$. We say that $\beta \in \mathcal{B}$ is \emph{carried} by (the conjugacy class of) 
a free factor $K$, if $\beta$ lies in the image of $\partial K \times \partial K \to B$. A subset, $S$, of $\mathcal{B}$ is carried by a free factor system, $\mathcal{F}$, if every element of $S$ is carried by some free factor in $\mathcal{F}$. 

An \emph{attracting lamination} for $\Phi$ is then a closed subset of 
$\mathcal{B}$ which is the closure of a single point, $\beta$, and has some extra properties ($\beta$ is birecurrent, admits an attracting neighbourhood for some positive iterate of $\Phi$, and is not carried by a $\Phi$-periodic free factor of rank one - see \cite[Definition 3.1.5]{Bestvina2000}). 

Any automorphism, $\Phi$, then admits a finite set of attracting laminations, $\mathcal{L}(\Phi)$ which is carried by a $\Phi$-invariant free factor system (\cite[Lemmas 3.1.6 and 3.1.13]{Bestvina2000}). Morally, one has one attracting lamination for every exponential stratum for a relative train track representative for $\Phi$ (or some power) and moreover $\mathcal{L}(\Phi)$ is canonical. 

For instance, a fully irreducible automorphism will have a single attracting lamination which is only carried by the whole group. A polyonomially growing automorphism will have no attracting laminations.

Again the theory of CT maps allows to compute this invariant. We would like to thank Mark Feign and Michael Handel for the proof of the following Proposition.

\begin{prop}\label{prop;tt_for_lamination_ffs}
    Given an automorphism, $\Phi \in \Out(F_n)$ of exponential growth, there is an algorithm to find the smallest free factor system which carries the set of attracting laminations. 
\end{prop}
\begin{proof}
    Since the set of attracting laminations is stable under taking positive powers, we are free to replace $\Phi$ with a (rotationless) power and use \cite{Feighn_Handel_2018} to algorithmically produce a CT map representing this power of $\Phi$. Since we are dealing with a rotationless automorphism, each exponential stratum produces exactly one attracting lamination. It is therefore sufficient to produce an algorithm which, given this CT map and a specific exponential stratum, $H_k$ of it, finds the smallest free factor carrying $\Lambda$, the associated attracting lamination. 

  Let $f : G \mapsto G$ be the CT map for the power of $\Phi$ and let $H_k$ be an exponential stratum with corresponding attracting lamination $\Lambda$. Construct a leaf $L$ of $\Lambda$ by choosing a point fixed by $f$ in the interior of an edge $e$ of $H_k$ and considering: 

  $$
e \subset f(e) \subset f^2(e) \subset \ldots 
$$
Note that $L$ is $f$-invariant.

Choose a segment of $L$ that starts and ends with a copy of the same oriented edge of $H_k$ and with $e$ interior to the segment. Glue the initial and terminal edges of the segment to form a loop $\gamma$ immersing to G. For $i \ge 0$, denote by $\gamma_i$ the immersed loop (or conjugacy class) $[f^i(\gamma)]$. Note that $\gamma_i$ is obtained from $L$ by gluing segments of the form $[f^i(edge)]$. In particular, these segments are long if $i$ is big.

\medskip

\noindent
\underline{Claim}: The free factor support $F(\Lambda)$ of $\Lambda$ is equal to the free factor support, $F(\cup_{i\ge 0} \{\gamma_i \})$ of 

$$
\cup_{i\ge 0} \{\gamma_i \}.
$$
Comment: The free factor support of a finite collection of elements can be found algorithmically - see \cite[Lemma 4.2]{Feighn_Handel_2018}. Moreover, sequence of free factor supports of $\{\gamma_1\}, \{\gamma_1, \gamma_2\}, \ldots $ stabilises. Hence, for an algorithm, it is enough to prove the claim.

\noindent
Proof of Claim: That $ F(\Lambda) \subset F(\cup_{i\ge 0} \{\gamma_i \})$  is clear.

For the other direction, we will show that $\gamma_i$ is contained in every $\Phi$-invariant free factor $F’$ that carries $\Lambda$. It is enough to assume that $i$ is large.

Let $G’$ be a marked graph with sub-graph $K’$ representing $F’$. Let $\tilde K’ \subset \tilde G’$ and $\tilde G$ denote universal covers. View $L$ (defined above) as a subset of $\tilde G$ such that its representation in $\tilde G’$ is a subset of $\tilde K’$. Let $T$ denote the covering translation of $G$ that identifies the segments of $L$ that give $\gamma_i$.

If lines in $\tilde G$ have long (depending on $G’$) overlap then their representations in $G’$ intersect - this is a consequence of the bounded cancellation lemma. Let $L’$ denote the representation in $\tilde G’$ of $L$ and $T’$ denote the covering translation of $\tilde G’$ corresponding to $T$. In particular, $L’$ and $T’(L’)$ intersect. Since $L’$ is in $\tilde K’$ and intersects $T’(L’)$ the union of $L’$ and $T’(L’)$ is contained in $\tilde K’$ - this follows since translates of $\tilde K’$ are either equal or disjoint.

More generally the union as $j$ ranges over the set of integers of $T’^j(L’)$ is connected, contained in $\tilde K’$, and $T’$-invariant. Hence $\gamma_i$ is carried by $F’$.
\end{proof}

\section{The case of a given invariant rank 2 free factor}\label{sec:rank2_reduction}

    For a free factor $K\leq F$ we define $$\Out(F,K) = \{ \Phi \in \Out(F) : \textrm{$K$ is $\Phi$-invariant}\}.$$ 

    For any $\Phi \in \Out(F,K)$, the restriction to $K$ is not well defined as an automorphism, but it is well defined as an outer-automorphism, because $K$ is its own normaliser in $F$. 
    
    This section is devoted to the proof of the following.

    \begin{thm}\label{thm:rank_2_reducible}
        Let $K\leq F_3$ be a free factor of rank 2. 
        Then conjugacy problem in $\Out(F_3,K)$  for pairs elements that induce infinite order outer-automorphisms of $K$, is decidable.

    \end{thm}

It will be enough to consider the case where $K = \langle a,b\rangle < F_3= \langle a, b, c\rangle$. Before the proof of Theorem~\ref{thm:rank_2_reducible} we estabish some needed result.

\begin{lem}
	\label{semidirect}
	Let $G=\Out(F_3,K)$, and  $\Chi \in G$. Then there exists a unique $\chi \in \Chi$ such that there exists  $\epsilon = \pm 1,  g \in \langle a,b \rangle$ and $\chi_0 \in \Aut(\langle a, b \rangle)$, satisfying

	\begin{equation}\label{eqn:K-inv}
	\chi: \left\{
	\begin{array}{ccl}
		c & \mapsto & g c^{\epsilon} \\
		b & \mapsto & b\chi_0 \\
		a & \mapsto & a \chi_0. \\
	\end{array} \right.
	\end{equation}
In particular, $G$ is isomorphic to an iterated semi-direct product, $C_2 \ltimes (\Aut(F_2) \ltimes F_2)$, where the $\Aut(F_2)$ action on $F_2$ is the natural one, and the $C_2$ action on $\Aut(F_2) \ltimes F_2$ is given by the involution which maps $(\chi_0, g)$ to $(\chi_0 \Ad(g), g^{-1})$.

\end{lem}	
			
	In the light of the decomposition of the lemma,  we can represent any  $\Out(F_3,K)$ uniquely as an ordered triple, $(\epsilon, \chi_0, g)$.  

\begin{proof}
    The outer automorphism $\Chi$ preserves the conjugacy class of $K = \langle a, b\rangle$, hence an element of $\Chi$ preserves $\langle a, b\rangle$; note that this element is only well-defined up to the normaliser of $\langle a, b\rangle$, which is again $\langle a, b\rangle$. 
    
    Then, $\Chi$ must send $c$ into $\langle a, b\rangle c^{\pm 1} \langle a, b\rangle$, and therefore (after composing with another inner automorphism by an element of $\langle a,b\rangle$) an element of $\Chi$ sends $c$ in $ \langle a, b\rangle c^{\epsilon}$, for $\epsilon \in \{-1, 1\}$. An inner automorphism preserving $\langle a, b\rangle$ and sending $c$ in $ \langle a, b\rangle c^{\pm 1}$ has to be trivial, and therefore we obtain the first part of the statement.  

    The set of $\Chi \in G$ for which $\epsilon = 1$ clearly forms a normal subgroup,  $G_1$. In $G_1$, the set of elements for which $\chi_0= Id$ is again a normal subgroup, $G_2$, isomorphic to $\langle a, b\rangle$. It is clear that the quotient $G_1/G_2$ (isomorphic to $\Aut(F_2)$) lifts in $G_1$, making $G_1 \simeq (\Aut(F_2) \ltimes F_2)$. Finally,       $G/G_1 \simeq \mathbb{Z}/2$ also clearly lifts in $G$, as the automorphism given by $\chi_0=1,  g=1, \epsilon = -1$, making $G$ an iterated semi-direct product. One verifies that the involution on $G_1$ given by the later lift is the one given in the statement.
\end{proof}

\begin{nt}~
	\begin{itemize}
		\item 	Let $G$ be a group and $H$ a subgroup. For $x,y \in G$ we write $x \sim_H y$ if there exists an $h \in H$ such that $x^h=y$. We do not require $x,y$ to be in $H$. 
		\item We say that $\sim_H$ is decidable if we can construct is an algorithm which decides on inputs $x,y \in G$, whether $x \sim_H y$. That is, $\sim_H$ is decidable means that it is a recursive subset of $G^2$, and we have an explicit algorithm to decide membership.
	\end{itemize}
\end{nt}	
	
\begin{lem}
	\label{finiteindex}
	Let $G$ be a group, $H$ a subgroup of $G$ and $H_0$ a finite index subgroup of $H$. Then if  $\sim_{H_0}$ is decidable and we can compute a complete set $h_1,\ldots,h_k$ of coset representatives of $H/H_0$ then so is $\sim_H$. 
\end{lem}
\begin{proof}
	Let $h_1, \ldots, h_k$ be computed set of coset representatives of $H_0$ in $H$. Then $x \sim_H y$ if and only if there exists an $i$ such that $x^{h_i} \sim_{H_0} y$.
\end{proof}

\begin{rem}
	Some care needs to be taken here. It is false - there are counterexamples - that if the conjugacy problem is solvable in a finite index subgroup, then it is solvable in the group. That is, there are examples of groups $G$, with finite index subgroups $H$ such that: 
	\begin{itemize}
		\item $\sim_H \cap H^2$ is recursive, 
		\item $\sim_G$ is not recursive. 
	\end{itemize}
	Hence conjugacy is solvable in $H$ but not in $G$. This doesn't contradict the Lemma above since we have the stronger hypothesis that $\sim_H$ is decidable.  
\end{rem}

The main tool we will use to solve conjugacy in $G$ is the twisted conjugacy problem. 

\begin{defn}
	Let $\phi \in \Aut(F_n)$. Then $x, y \in F_n$ are said to be twisted $\phi$ conjugate, written $x \sim_{\phi} y$ if there exists a $w \in F_n$ such that:
	$$
	x = (w \phi) y w^{-1}.
	$$
\end{defn}

We need the following later: 
\begin{lem}
	\label{powers}
	Let $\phi \in \Aut(F_n)$ and $x \in F_n$. Then for any $k \in \Z$, $x \sim_{\phi} x \phi^k$.
\end{lem}
\begin{proof}
	Simply notice that $x\phi = (x \phi) x x^{-1}$ and $x \phi^{-1} =   ((x^{-1} \phi^{-1}) \phi ) x (x \phi^{-1})$ and that $\sim_{\phi}$ is an equivalence relation.   
\end{proof}

\begin{thm}[Theorem 1.5 of \cite{BMMV}]
	\label{twistedconj}
	Let $\phi \in \Aut(F_n)$. Then the twisted conjugacy problem for $\phi$ is solvable. That is, $\sim_{\phi}$ is a recursive subset of $F_n^2$. 
\end{thm}

We will also need the following: 

\begin{thm}
	\label{conjandcentraliser}
The conjugacy problem is solvable in $\Out(F_2)$ and $\Aut(F_2)$. Moreover for any $\Phi \in Out(F_2)$ of infinite order and for any $\phi\in\Phi$:
\begin{itemize}
    \item The centraliser of $\Phi$ in $\Out(F_2)$ is virtually cyclic;
    \item there is an algorithm that computes the coset representatives of $\langle \Phi\rangle$ in its centraliser;
    \item  the centraliser of $\phi$ in $\Aut(F_2)$ has a finite index subgroup  $$C_0=\{ \phi^k \Ad(g) \ : \ k \in \Z, \ g \in Fix(\phi)\} = \langle \phi \rangle \times \langle \Ad(g) \ : \ g \in Fix(\phi) \rangle;$$
    \item the group $D=\{ \chi \in \Aut(F_2) \ : \ \phi^\chi = \phi \Ad(h), h \in F_2\}$ has a finite index subgroup  $D_0 = \{ \phi^k \Ad(x) \ : \ k \in \Z, x \in F_2 \}$. 
\end{itemize}

\end{thm}

\begin{proof}
    We note that $\Out(F_2) \cong GL_2(\Z)$ is virtually a free group of rank 2, and therefore word hyperbolic. Therefore the conjugacy problem is solvable and centralisers of infinite order elements are virtually cyclic. For any element of infinite order, $g$ in a hyperbolic group it is known to be algorithmic to find the coset representatives of $\langle g \rangle$ in its centraliser. This is folklore, but a detailed proof appears in Proposition 4.11 of \cite{BMV}. 

    The solution for the conjugacy problem for $\Aut(F_2)$ appears in \cite{Bogopolski2000} and also as Corollary 5.2 of \cite{BMV}. 

    The rest of the statements about $\Aut(F_2)$ are then just translations of the corresponding statements about $\Out(F_2)$. 
\end{proof}

We now address the conjugacy problem in $\Out(F_3,K)$ for elements that induce infinite order outer-automorphisms on $K$.

\begin{proof}[Proof of Theorem \ref{thm:rank_2_reducible}]
Let $G=\Out(F_3,K)$, and $\Phi, \Psi \in G$, given by data, as in Lemma~\ref{semidirect}, as \[\Phi=(\epsilon_1, \phi_0, u)\textrm{ ~and~ }\Psi=  (\epsilon_2, \psi_0, v).\] 

We assume that $\phi_0, \psi_0$ define infinite order outer-automorphisms of $K$. 
We must decide whether  there exists a conjugator, $\chi= (\epsilon_3, \chi_0, h) \in G$, such that $\Phi^{\chi} = \Psi$. By Lemma~\ref{finiteindex}, we may assume that $\epsilon_3 =1$. We calculate the possibilities for $\Phi^\chi$, listed as a triple:

\begin{equation}
\label{conjugate}	
\begin{array}{ll}
	\mbox{If } \epsilon_1=1 & \Phi^\chi = (1, \phi_0^{\chi_0}, (h^{-1} \phi_0^{\chi_0}) (u \chi_0) h) \\ \\
	\mbox{If } \epsilon_1=-1 & \Phi^\chi = (-1, \phi_0^{\chi_0} \Ad(h), h^{-1} (h^{-1} \phi_0^{\chi_0}) (u \chi_0) ) \\
\end{array}
\end{equation}

Hence if $\epsilon_1 =1$, then for $\Phi$ and $\Psi$ to be conjugate, we would require that $\phi_0$ and $\psi_0$ are conjugate. If 	$\epsilon_1 =-1$, they would need to be conjugate as outer automorphisms. But notice that in this case, when $\epsilon_1 =-1$, we can set $\chi_0=Id$, and by varying $h$, we can obtain triples where the second entry is anything of the form $\phi_0 \Ad(h)$. 

Hence, if $\Phi, \Psi$ are to be conjugate, we would require $\epsilon_1 = \epsilon_2$ (otherwise we know that they are not conjugate). Moreover, by solving conjugacy in $\Aut(F_2)$ or $\Out(F_2)$, we may assume that $\phi_0=\psi_0$ as automorphisms. To summarise, we have reduced the problem to the situation where: 

\begin{itemize}
	\item $\epsilon_1 = \epsilon_2$, 
	\item $\phi_0 = \psi_0$ in $\Aut(F_2)$. 
\end{itemize}

\noindent
\underline{Case 1: $\epsilon_1 = \epsilon_2 =1$. }: 

In this case we must have that $\phi_0^{\chi_0} = \psi_0 = \phi_0$. Hence $\chi_0$ centralises $\phi_0$. Let $C=C_{\Aut(F_2)}(\phi_0)$ be the centraliser of $\phi_0$ in $\Aut(F_2)$. Then $\chi$ lies in the subgroup, $C \ltimes F_2$ of $G$. By Lemma~\ref{finiteindex}, it will be enough to solve the problem where we look at conjugators that lie in the subgroup, $C_0 \ltimes F_2$, where $C_0$ is the finite index subgroup of $C$ given by Theorem~\ref{conjandcentraliser}. Hence we may assume that $\chi_0 = \phi_0^k \Ad(x)$, where $k \in \Z$ and $x \in Fix(\phi_0)$. 

But then we get that: 
$$
 \Phi^\chi = (1, \phi_0, (h^{-1} \phi_0) (u \phi_0^k \Ad(x)) h) = (1, \phi_0, ((h^{-1} x^{-1}) \phi_0) (u \phi_0^k) xh),
$$
using that fact that $x \phi_0 = x$. Since $\Psi = (1, \phi_0, v)$, that means we are trying to decide whether there exist $k \in \Z$, $h \in F_2$, $x \in Fix(\phi_0)$ such that: 

$$
(h^{-1} x^{-1}) \phi_0 (u \phi_0^k) xh = v.
$$
Putting $h'=xh$, this is equivalent to deciding: 
$$
({h'}^{-1}) \phi_0 (u \phi_0^k) h' = v. 
$$
But by Lemma~\ref{powers}, this is equivalent to saying that $u \sim_{\phi_0} v$, which is decidable by Theorem~\ref{twistedconj}.

\noindent
\underline{Case 2: $\epsilon_1 = \epsilon_2 =-1$. }: 

As before, we have that $\phi_0 = \psi_0$. This time, by Equation~\ref{conjugate}, we get that $\chi_0$ belongs to the subgroup $D$ given by Theorem~\ref{conjandcentraliser}. Hence, by Lemma~\ref{finiteindex}, we may assume that $\chi_0 \in D_0$. Hence, $\chi_0 = \phi_0^k \Ad(x)$ for some $x \in F_2$. Therefore, using Equation~\ref{conjugate} and simplifying: 
$$
 \Phi^\chi = (-1, \phi_0 \Ad(w), h^{-1} (h^{-1} \phi_0 \Ad(w h^{-1})) (u \chi_0) ),  
$$
where $w = (x^{-1} \phi_0)xh$. However, since we have reduced to the case where $\phi_0= \psi_0$, we get the extra restriction that $w=1$ that is equivalent to $h=x^{-1} (x \phi_0)$. And hence for $\Phi$ and $\Psi$ to be conjugate we would need to simply equate the third entry of the triple: 
$$
(h^{-1} \phi_0) h^{-1} x^{-1} (u \phi_0^k) x = v.
$$
Putting  $h=x^{-1} (x \phi_0)$ this gives: 
$$
(x^{-1} \phi_0^2) (x \phi_0) (x^{-1}  \phi_0) (u \phi_0^k) x = (x^{-1} \phi_0^2) (u \phi_0^k) x = v.
$$
Hence this problem is equivalent, by Lemma~\ref{powers} to $u \sim_{\phi_0^2} v$ or $u \phi_0 \sim_{\phi_0^2} v$, both of which are solvable by Theorem~\ref{twistedconj}. 
\end{proof}

\begin{rem}\label{rem:no_induction}
    A natural question to ask is whether the strategy above can be generalized to higher rank free groups. We note that ingredients used in the proof above are
    \begin{itemize}
        \item \emph{The conjugacy problem in $\Aut(F_2)$.} The conjugacy problem in $\Aut(F_n)$ appears to be harder than the conjugacy problem in $\Out(F_n)$.
        \item \emph{Explicit computations of centralizers in $\Out(F_2)$.} In our case, we were able to exploit the fact that $\Out(F_2)$ is virtually free. It was shown in \cite{bestvina_laminations_1997} that centralizers of irreducible automorphisms are virtually cyclic. As for more general automorphisms, the current state of the art for polynomially growing automorphisms \cite{rodenhausen_centralisers_2013,rodenhausen_centralisers_2015,andrew_centralisers_2022} amounts to establishing finiteness properties.
    \end{itemize}
    For these reasons, the argument above does not obviously generalize.
\end{rem}

\section{Polynomially growing automorphisms}\label{sec:quadratic}

Before turning our attention to the quadratic growth case in $F_3$ (i.e. the largest polynomial growth permitted there), let us record the combination of  classical results, on the conjugacy problem for low degree polynomial growth outer-automorphisms of $F_n$.

\begin{thm}[\cite{krstic_actions_1989,KrsticLustigVogtmann}]\label{thm:CP_subquadratic}
    The conjugacy problem in $\Out(F_n)$ is decidable among outer automorphisms with growth of polynomial degree 0 or 1.
\end{thm}

Recall also that the conjugacy problem among unipotent polynomially growing outer automorphisms  of $F_n$ has been solved by \cite{Feighn_Handel_Conjugacy}.

\subsection*{Quadratic polynomial growth in $F_3$.}

Every $\Phi \in  \Out(F_n)$ of polynomial growth has a power $\Phi^n$ that is unipotent. In fact, one can take this power to be uniform for every $n$; the exponent (or order) of $GL_n(\Z_3)$ suffices, as one can define a UPG automorphism as one of polynomial growth which induces the trivial map in $\Z_3$ homology. 

\begin{defn}
	Two automorphisms $\phi, \psi \in \Aut(F_n)$ are said to be\textit{ isogredient }if they are conjugate by an inner automorphism. 
\end{defn} 
    Recall that any outer-automorphism of $F_n$ is a  coset of $\Inn(F_n)$ in $\Aut(F_n)$.  On each such outer automorphism, the relation of isogredience is an equivalence relation.

\begin{thm}[Bestvina-Handel Theorem, \cite{Bestvina1992}]
	\label{Bestvina-Handel}
	Let $\Phi \in \Out(F_n)$. Then,
	\[
	\sum \max\{ \rank(\fix) - 1, 0 \} \leq n-1,
	\]
	where the sum is taken over representatives, $\phi$, of isogredience classes in $\Phi$. 
\end{thm}

\begin{thm}
	\label{quadraticBH}
		Let $\Phi \in \Out(F_n)$ have quadratic growth. Then, 
			\[
		1 \leq \sum \max\{ \rank(\fix) - 1, 0 \} \leq n-2,
		\]
		If $n=3$, this sum is exactly equal to 1 and exactly one isogredience class has a non-zero contribution to this sum. 
\end{thm}
\begin{proof}
	This follows from results in either \cite{Levitt2009} or \cite{martino_maximal_2002}.
 A full discussion of this is also given in \cite{Andrew2022}. 

 More specifically, as argued in \cite[Theorems 2.4.6 and 2.4.8]{Andrew2022}, any outer automorphism, $\Phi$ of $F_n$, which satisfies the equality from \ref{Bestvina-Handel}, namely: 

 \[
	\sum \max\{ \rank(\fix) - 1, 0 \} = n-1,
	\]
 must have linear growth. 
\end{proof}

\begin{lem}
	\label{rank2upg}
	Let $\Phi \in \Out(F_2)$ be UPG. Then some automorphism in $\Phi$ has fixed subgroup of rank 2. 
\end{lem}
\begin{proof}
	It is easy to show that there is a basis of $F_2$, $\langle a,b \rangle$ where (some automorphism of) $\Phi$ acts as: $a \mapsto a, b \mapsto b a^k$ for some integer $k$. Hence the fixed subgroup is  $\langle a, bab^{-1} \rangle$ or $\langle a, b \rangle$ when $k=0$. 
\end{proof}

\begin{thm}
	\label{UPGinvariance}
	Let $\Phi \in \Out(F_3)$ be a UPG automorphism of quadratic growth. Then $\Phi$ admits a rank 2 invariant free factor, $K$ which is algorithmically computable. Moreover, $K$ is unique in the sense that if $K_1$ is any other invariant rank 2 free factor, then $K_1$ is conjugate to $K$. 
\end{thm}
\begin{proof}
	By \cite{Bestvina2005}, there is a basis, $a,b,c$ for $\Phi$ such that some $\phi \in \Phi$ has the following representation. In fact, by \cite{Feighn_Handel_2018} there is an algorithm to produce the following basis: 
	
	$$
	\begin{array}{rcl}
		& \phi & \\
		a & \mapsto &  a \\
		b & \mapsto & ba^k \\
		c & \mapsto & u c v,
	\end{array}
	$$ 
	where $k \in \Z$ and $u, v \in \langle a,b \rangle$ and $\phi \in \Phi$.

 Notice that $k=0$ implies that $\Phi$ has linear growth, and hence we deduce that $k \neq 0$. In particular, $\fix = \langle a, bab^{-1} \rangle$ (it cannot have higher rank due to Theorems~\ref{Bestvina-Handel} and \ref{quadraticBH}). Hence $\langle a,b \rangle$ is the smallest free factor containing $\fix$.

Thus we have algorithmically produced a $\Phi$-invariant rank 2 free factor, $\langle a,b \rangle$, and all that remains is to show is that it is unique. 
	
	If $K_1$ were another $\Phi$-invariant rank 2 free factor, then the restriction of $\Phi$ to $K_1$ would again be UPG, and Lemma~\ref{rank2upg} would imply that some $\psi \in \Phi$ would have a fixed subgroup of rank 2, contained in $K_1$. But Theorem~\ref{quadraticBH} then implies that $\phi$ and $\psi$ are isogredient and, hence have conjugate fixed subgroups implying that $\langle a,b \rangle$ and $K_1$ are also conjugate. 
\end{proof}

\begin{cor}\label{cor:quadratic_reducible}
	Let $\Phi \in \Out(F_3)$ be an automorphism of quadratic growth. Then $\Phi$ admits a unique invariant rank 2 free factor. Moreover, there is an algorithm to determine this free factor. 
\end{cor}

\begin{proof}
	Some positive power, $\Phi^k$ of $\Phi$ is UPG. By Theorem~\ref{UPGinvariance}, $\Phi^k$ admits a unique invariant rank 2 free factor, $K$. But $K \Phi$ is also $\Phi^k$ invariant, and hence the uniqueness of $K$ implies that $K$ is  conjugate to $K \Phi$. 
	
	Finally, if $K_1$ is any $\Phi$-invariant rank 2 free factor, it must also be $\Phi^k$-invariant, and hence by Theorem~\ref{UPGinvariance} again, $K_1$ is conjugate to $K$. 

 This free factor is produced algorithmically as in Theorem~\ref{UPGinvariance}. 
\end{proof}

\subsection*{Conjugacy Problem in polynomial growth for  $\Out{F_3}$.}

We may conclude for this section. 

Recall (Proposition \ref{prop:alggrowth})  that given $\Phi$ and $\Psi$ two outer-automorphisms of $F_3$ we may decide whether they are of polynomial growth, and we may compute their degree.

Assume first that both are polynomial growth of degree $2$. We may compute by \ref{cor:quadratic_reducible} the unique conjugacy classes $[K_\Phi], [K_\Psi]$ of invariant free factor of rank 2 of $F_3$. After conjugating $\Psi$ by a (computable) automorphism that sends $K_\Phi$ to $K_\Psi$, we may assume that these two groups are equal. We may then apply Theorem \ref{thm:rank_2_reducible} in order to decide whether $\Phi$ and $\Psi$ are conjugate in $\Out{(F_3,K)}$.  

Since by Corollary \ref{cor:quadratic_reducible}, the rank 2 invariant free factor of $\Phi$ is unique, the two outer-automorphisms are conjugate in $\Out{(F_3,K)}$ if and only if they are conjugate in $\Out{(F_3)}$.  

This, together with Theorem \ref{thm:CP_subquadratic}, solves the conjugacy problem in $\Out{F_3}$ for all polynomially growing outer-automorphisms.

 \section{Exponential growth: laminations}\label{sec:expo_i}

Recall (references in Section \ref{sec;decide_irred}, Proposition \ref{prop:alggrowth}) that one can decide whether an outer-automorphism is of exponential growth, and whether it is irreducible. In that case, we recall the following. 
 
   \begin{thm}  
    [{\cite{Los, Lustig2007,Kapovich_algo, Francaviglia2022}}] 
    \label{thm:CP-iwip}

     The conjugacy problem is decidable among irreducible elements of $\Out(F_n)$.  
     
    \end{thm}

We will now focus on the reducible case.

\subsection*{Invariant Free Factor Systems in $F_3$}
    We will now work towards describing reducible exponentially growing automorphisms. Following \cite{Bestvina2000} we have

\begin{prop}
	\label{oneexp}
	Let $\Phi \in \Out(F_3)$. Then any relative train track representative of $\Phi$ has at most one exponential stratum. 
\end{prop}
\begin{proof}
	Suppose that $f: \Gamma \to \Gamma$ is a relative train track representative of $\Phi$. Suppose that $H_r$ is an exponential stratum and that $G_r$ is the corresponding $f$-invariant subgraph (the union of all the strata, up to and including $H_r$). 
	
	Similarly, let $G_{r-1}$ be the $f$-invariant subgraph consisting of the union of the strata up to, but not including, $H_r$. 
	
	Let $C_r$ be a connected component of $H_r$ and let $C_{r-1}$ be a connected component of $C_r \cap G_{r-1}$ (which we allow to be empty). Since the transition matrix of $H_r$ is irreducible, some power of $f$ must leave both $C_r$ and $C_{r-1}$ invariant. However, since $H_r$ is an exponential stratum, the the rank of $\pi_1(C_r)$ is at least 2 more than the rank of $\pi_1(C_{r-1})$. (For instance, take an edge of $H_r$ incident to $C_{r-1}$ an take a power, $k$, of $f$ such that $f^k(e)$ crosses $e$ at least 3 times). This immediately gives the result: either the rank of $\pi_1(C_{r-1})$ is zero and the rank of $\pi_1(C_{r})$ is at least 2, in which case there can only be zero or polynomial strata above $H_r$, or the rank of $\pi_1(C_{r-1})$ is one, and there can be no strata above $H_r$. 
\end{proof}

\begin{cor}\label{cor:invariant_carrier}
	Let $\Phi \in \Out(F_3)$ have exponential growth. Then $\Phi$ has exactly one attracting lamination. This lamination is carried by a $\Phi$-invariant conjugacy class of a free factor, $K$ of $F_3$. 
	
	The rank of $K$ is either 2 or 3. If the rank of $K$ is 2, then it is unique in the following sense: if $K_1$ is another free factor of rank 2 which is $\Phi$-invariant (up to conjugacy), then $K_1$ is conjugate to $K$.
	
\end{cor}

\begin{proof}
	Some positive power of $\Phi$ has an improved relative train track representative by~\cite[Theorem 5.1.5]{Bestvina2000}. All we will really need here is that this relative train track is eg-aperiodic; that is, the transition matrix of every exponential stratum is not just irreducible but also aperiodic (also known as primitive), meaning that it has a power where every entry is positive. In fact, it is clear that every relative train track map has a positive iterate which is eg-aperiodic, and we don't really need the full force of \cite[Theorem 5.1.5]{Bestvina2000}. 
	
	Let us call this relative track representative, $f: \Gamma \to \Gamma$ and denote the $r^{th}$ stratum by $H_r$. 
	
	Then, by \cite[Lemma 3.1.10]{Bestvina2000}, if $\beta$ is a generic line of some lamination for $\Phi$, then representing $\beta$ in $\Gamma$ yields a line, $\lambda$, where the highest stratum crossed by $\lambda$ is an exponential stratum. (Using the terminology from \cite{Bestvina2000}). Moreover, by~\cite[Corollary 3.1.11]{Bestvina2000}, since our relative train track is eg-aperiodic, any generic line of any lamination whose realisation in $\Gamma$ which crosses the same $H_r$ but no higher stratum will have the same closure as  $\beta$. However, by Proposition~\ref{oneexp}, $f$ only admits one exponential stratum. Since laminations are closures of such lines this means that there is exactly one lamination for $\Phi$, which must therefore be fixed by $\Phi$ (and not just periodic for $\Phi$). 
	
	As in \cite[Corollary 2.6.5 and Definition 3.2.3]{Bestvina2000}, there is a unique free factor whose conjugacy class carries the lamination of $\Phi$. Since the lamination is fixed by $\Phi$, this free factor is $\Phi$-invariant up to conjugacy.  
	
	For the final part, we note that an exponential stratum cannot have a component which has rank 1 as a graph (is homotopic to a circle), so $K$ must have rank 2 or 3.

	Furthermore, if $K$ has rank 2, and if $K_1$ is a $\Phi$-invariant free factor (up to conjugacy) also of rank 2, then we can form - via \cite[Theorem 5.1.5]{Bestvina2000} - an improved relative train track representative for some power of $\Phi$, where $K_1$ is the fundamental group of some invariant subgraph, $G_r$. 
 If the restriction of the relative train track map to $G_r$ is polynomial, then the whole automorphism will have polynomial growth; this is excluded by hypothesis. Hence, by~\cite[Lemma~3.1.9]{Bestvina2000}, there is a lamination carried by (the conjugacy class of) $K_1$. But there is only one lamination for $\Phi$, hence $K$ and $K_1$ are conjugate. 
		
\end{proof}

\subsection*{Conjugacy problem for exponential growth with rank $2$ lamination}

Consider $\Phi$ and $\Psi$ two outer-automorphisms of $F_3$ that are of exponential growth, and whose attracting lamination is carried by a free factor of rank $2$, respectively $K_\Phi$ and $K_\Psi$.  

 By Proposition \ref{prop;tt_for_lamination_ffs}  (and the unicity of Corollary \ref{cor:invariant_carrier}), the groups $K_\Phi$ and $K_\Psi$ can be computed. 

 Since both $K_\Phi$ and $K_\Psi$ are free factors of same rank, there exists (and one can compute) an automorphism $\chi$ of $F_3$ sending $K_\Psi$ to $K_\Phi$, and  after conjugating $\Psi$ by  the outer-class  $\Chi$ of $\chi$, we may assume that $K_\Phi = K_\Psi$, and we denote it $K$. 

Since by Corollary \ref{cor:invariant_carrier}, this invariant free factor is unique, the two automorphisms  $\Phi$ and $\Psi$ are conjugated in $\Out(F_3)$ if and only if they are conjugated in $\Out(F_3, K)$. This can be decided by Theorem \ref{thm:rank_2_reducible}.

\section{Exponential growth: mapping tori}\label{sec:expo_ii}

In this section, we 
 take the point of view of analysing the  semi-direct products of $F_3$ that are associated to automorphisms (see also \cite{Sela, Dahmani2016, Dahmani2017, DahmaniTouikan2021}). Although this point of view allows to treat the conjugacy problem for all the exponentially growing outer-automorphisms  of $F_3$, we restrict our presentation to the remaining case in Flowchart \ref{fig:THE_algo}, namely the case of automorphisms whose attracting lamination carrier is the entire group $F_3$.

Given $\phi \in \Aut(F_n)$, the associated mapping torus  is $F_n\semidirect_\phi \langle t \rangle$. The normal subgroup $F_n < F_n\semidirect_\phi \langle t \rangle$ is called a \emph{fiber}.  If $\phi = \psi\Ad(g)$ then we have  $F_n\semidirect_{\psi}\langle t \rangle= F_n\semidirect_{\phi}\langle tg \rangle$, in particular $\Phi=[\phi]\in\Out(F_n)$ has a well-defined mapping torus. The following proposition describes how it relates to conjugacy in $\Out(F_n)$. 
We refer the reader to  \cite{DahmaniTouikan2021} for precise definititions.
    
\begin{prop}[Standard mapping torus criterion for conjugacy]\label{prop;criterion_fo_isom}
       Let $G$ be a group,  $\Phi,\Psi \in \Out(G)$ and  $\phi \in \Phi, \psi \in \Psi$. Then $\Phi$ is conjugate to $\Psi$ in $\Out(G)$ if and only if there is an isomorphism\[
        f: G \semidirect_\phi \langle t \rangle \to G \semidirect_\psi \langle t \rangle
        \] such that $f(G) = G$ and $f(t) = tw$ for some $w \in G$.
\end{prop}

The following theorem relates dynamical characteristics of outer automorphisms to the structure of their mapping tori.

\begin{thm}[Relative hyperbolicity in exponential growth]\label{thm;relhypcori}
    For any $\Phi\in \Out(F_n)$, its mapping torus admits a properly relatively hyperbolic (possibly word-hyperbolic) metric if, and only if, $\Phi$ has exponential growth. Peripheral subgroups of the relatively hyperbolic structure can be taken to be the mapping tori of $\Phi$ restricted to maximal polynomially growing subgroups.
\end{thm} 

\begin{proof} 
    The converse implication was obtained in  \cite[Theorem 3.1]{ghosh_relative_2023} and \cite[Theorem 4]{DamaniLi2020}. The direct implication is found in  \cite{macura_detour_2002}, \cite[Prop 1.3]{Dahmani2017}   see also \cite{hagen2019},  \cite[Theorem 1.2]{Brinkmann2000}.
\end{proof}

Using this strategy, in \cite{DahmaniTouikan2021} the conjugacy problem in $\Out(F_n)$  is completely reduced  to specific algorithmic problems in the peripheral subgroups.  These problems are the algorithmic tractability for their subgroups (effective coherence, conjugacy problem, generation problem), the Minkowski property for certain subgroups, the mixed Whitehead problem, and the conjugacy problem for the induced automorphisms on maximal polynomially growing subgroups.  In this reduction, the  exponential growth part of the outer automorphisms is  completely evacuated from the discussion (it is treated during the reduction).

In the case of exponentially growing automorphisms of $F_3$, the polynomially growing subgroups are sufficiently small that we can complete a solution of the conjugacy problem in their case. We will explain this in the remaining case of Flowchart \ref{fig:THE_algo}, in which the polynomially growing subgroups are even simpler. 

This will require two steps. The first step is giving the solution to the criterion of Proposition \ref{prop;criterion_fo_isom} in the case the mapping tori of the given autmorphisms are so called almost toral. The second step  is proving that if the mapping torus of an automorphism is not almost toral, then there is a rank 2 free factor carrying the attracting lamination.  This allows to conclude since this later case was already treated. Treating it alternatively through the criterion of Proposition \ref{prop;criterion_fo_isom} is still possible, and involves cases covered by the larger study \cite{DahmaniTouikan2023}, but is not done here.

\subsection*{The Almost Toral case}

We say $\Phi \in \Out(F_n)$ or its mapping torus $F_n \rtimes_\phi \mathbb Z, \phi \in \Phi$ is \emph{almost toral relatively hyperbolic} if $F_n \rtimes_\phi \mathbb Z$ is hyperbolic relative to a collection of subgroups isomorphic to $\mathbb Z \times \mathbb Z$ or to $\mathbb Z \rtimes \mathbb Z$, the fundamental group of a Klein bottle. By extension, in that case, we say that the automorphism is almost toral relatively hyperbolic.

The groups $\mathbb Z \times \mathbb Z$ and $\mathbb Z \rtimes \mathbb Z$ are the only possible virtually abelian peripheral subgroups in the relatively hyperbolic structure of Theorem \ref{thm;relhypcori}. Actually the following is immediate from Theorem \ref{thm;relhypcori} and Nielsen-Schreier theorem.

\begin{prop}
    The mapping torus of an outer automorphism of $F_n$ is almost toral if and only if its maximal polynomially growing subgroups have rank one.
\end{prop}

It is furthermore decidable whether the mapping torus of a given automorphism of $F_n$ is almost toral, by \cite{DahmaniGuirardel2013} (this is decidable in the following sense:  there is an algorithm that will terminate if the automorphism is exponentially growing, and provide presentations for each conjugacy representative of peripheral subgroup, and indicate whether or not  they are abelian, or isomorphic to $\mathbb{Z}\rtimes \mathbb{Z}$).

\begin{thm}\label{thm;solve-atrh}
    The conjugacy problem in $\Out(F_n)$ is decidable among the almost toral relatively hyperbolic automorphisms.
\end{thm}

\begin{proof}
    If the mapping tori are hyperbolic (i.e. if the peripheral structures are empty) it is enough to invoque \cite{Dahmani2016}. We now assume that the peripheral structures are non-empty.

    Using \cite[Theorem 2.1]{DahmaniTouikan2021}, it suffices to check that $\mathbb{Z}^2$ and  $\mathbb Z \rtimes \mathbb Z$ form a class of algorithmically tractable groups, with solvable fiber-and-orientation preserving mixed Whitehead problem, and Minkowski property for their subgroups. 

    The algorithmic tractability is actually immediate, despite the definition of this property.

    That the group $\mathbb{Z}^2$ (or any of its subgroups) satisfies the Minkowski property is due to the classic observation of Minkowski: $GL_2(\mathbb{Z}) \to GL_2(\mathbb{Z}/3) $ has torsion free kernel.

    For $K=\mathbb Z \rtimes \mathbb Z$, the Minkowski property asks for a finite characteristic quotient in which all torsion elements of $\Out( \mathbb Z \rtimes \mathbb Z)$ survive. Let us write $K= \langle a\rangle \rtimes \langle t \rangle$. The following lemma establishes the Minkowski property.

    \begin{lem}\label{lem;Mink_K}
        The abelianisation $K^{ab}$ of $K$ is $K/\langle a^2 \rangle \simeq (\mathbb{Z}/2) \times \mathbb{Z}$.
        
        $\Out(K) \simeq (\mathbb{Z}/2) \times (\mathbb{Z}/2)$, and injects by the congruence map in $\Out( (\mathbb{Z}/2) \times (\mathbb{Z}/3 ))$.
    \end{lem}
    
    \begin{proof}
        The first assertion is standard. The subgroups $\langle t^2 \rangle$ (the center) and $\langle a \rangle $ (the preimage of the torsion in the abelianisation) are preserved by all automorphisms. The square roots of $t^2$ are the $a^k t$, $k\in \mathbb{Z}$.  Therefore any automorphism has the form 
        $(a\mapsto a^{\epsilon_1}, t\mapsto a^kt^{\epsilon_2})$ for $\epsilon_1, \epsilon_2 \in \{-1, 1\}$, and $k\in \mathbb{Z}$. Since $(a\mapsto a^{-1}, t\mapsto t)$, and $(a\mapsto a, t\mapsto a^2t)$ 
        are both inner, it follows that $\Out(K) \simeq (\mathbb{Z}/2) \times (\mathbb{Z}/2)$, with representatives being $(a\mapsto a, t\mapsto a^ut^{\epsilon})$ for $u\in \{0, 1\}$, and $\epsilon \in \{-1,1\}$.  All three non-trivial automorphisms descend as obviously non trivial automorphisms in $\Aut((\mathbb{Z}/2) \times (\mathbb{Z}/3 ))$, when quotienting by $a^2$ and $t^3$. 
    \end{proof}

    It remains to solve the fiber and orientation preserving mixed Whitehead problem for $K$, the fiber being $\langle a\rangle$. This problem is an orbit problem for the group of outer automorphisms preserving fiber and orientation, for tuples
    of conjugacy classes of tuples. In the case of $K$, all outer automorphisms of $K$ (see their expression above) preserve the fiber $\langle a \rangle$ and the orientation $\langle a \rangle t$. There are only four of them, so the problem reduces to the classical multiple conjugacy problem in $K$, which is easy to solve.

\end{proof}

\subsection*{The polynomially growing subgroups in $F_3$}

By Levitt's theorem \cite[Theorem 4.1]{Levitt2009},   given an outer automorphism of $F_n$, its maximal polynomially growing subgroups have rank $\leq n$,  they have rank $<n$ if the automorphism has an exponentially growing conjugacy class, and in the later case,   if one such group  has rank $n-1$, it is unique up to conjugacy. 

In the case of $n=3$,  the possible polynomially growing subgroups for an exponentially growing outer automorphism of $F_3$ are then either trivial, or cyclic, or of rank $2$. Given the previous discussion, we consider the case of a single conjugacy class of  polynomially growing subgroup of rank $2$. It turns out, as we will show,  that it must be placed very specifically in the group $F_3$, revealing a lamination carrier of rank 2.

Recall that if $D$ is a $\Phi$-invariant subgroup, we can find $\phi \in \Phi$ such that $\phi(D)=D.$

    \begin{prop} \label{prop;poly_rk_2_implies_holed_torus} 
        If $\phi \in \Aut (F_3)$ is of exponential growth and has an invariant subgroup $D$ of rank at least $2$ on which it induces a polynomially growing automorphism,  then it preserves the conjugacy class of a rank $2$ free factor $Q$ on which it induces an exponentially growing outer-automorphism. 

        Moreover, the conjugacy class of such a free factor is unique and there exists a conjugate $D'$ of $D$, containing a subgroup $C'$ of $Q \cap D'$, such that $C'$ is generated by the commutator of a basis in $Q$, and $F_3= Q*_{C'}D'$.
    \end{prop}

    \begin{proof}
        We assume without loss of generality that $D$ is a maximal $\phi$-invariant subgroup on which $\phi|_D$ is polynomially growing. By maximality, we cannot have $D <H <F_3$ with $H$ of rank 2 and $\phi$-invariant. Indeed, if this were the case either $\phi|_H$ is either polynomially growing, which contradicts maximality, or exponentially growing with a polynomially growing invariant subgroup of rank 2, which is impossible in $\Out(F_2)$. Furthermore, $D$ cannot be a free factor of $F_3$ as the xpression \eqref{eqn:K-inv} from Lemma \ref{semidirect} would imply that $\phi$ is also polynomially growing. It follows that $(F_3,D)$ is relatively one-ended.

        By \cite[Theorem II.2]{GJLL} $F_3$ acts faithfully on an $\mathbb R$-tree $T_\infty$ with trivial arc stabilizers and there is a homotethy $H:T_\infty \to T_\infty$ such that for all $x\in T_\infty$ and $f \in F_3$, $f\cdot H(x) = H((f\phi) x)$ (here, a homotethy is a map satisfying $d(H(x),H(y)) = \lambda d(x,y)$ for some fixed stretch factor $\lambda\geq 1$). Note that from \cite[\S B - \S E]{GJLL} the action of $F_3$ on $T_\infty$ is obtained as a limit of rescaled actions on a fixed free cocompact action of $F_3$ on a simplicial metric tree $\tau$. Exponential growth of $\phi$ implies that $D$ fixes a point in $T_\infty$. %

        Trivial arc stabilizers and the non-trivial subgroup $D$ (which cannot be in a proper free factor) that fixes a point (so the action is not free) imply that we may apply \cite[Lemma 4.6]{Horbez_Tits} to the action of $F_3$ on $T_\infty$ with the empty free factor system. This lemma gives three possiblities. Two of these imply that point stabilizers must all either be cyclic or contained in proper free factors of $F_3$, which is impossible because of the properties of $D$. The remaining possiblity is that the action of $F_3$ on $T_\infty$ has a so-called dynamical proper free factor. This implies that $T_\infty$ is not a  simplicial $\mathbb R$-tree, since by definition a dynamical proper free factor must act on its minimal invariant tree with dense orbits.

        We now apply \cite[Theorem 5.1]{guirardel_actions_2008}. The triviality of arc stabilizers, one-endeness of $F_3$ relative to $D$, and the faithfullness of the action of $F_3$ on $T_\infty$, imply that the action of $F_3$ on $T_\infty$ decomposes into a graph of actions where each vertex action is either simpicial, Seifert type, or axial. Free groups cannot admit faithful axial actions and since $T_\infty$ is not simplicial, the graph of actions must contain a Seifert type vertex. Orbifolds with free fundamental groups are surfaces with boundary. Therefore $F_3$ decomposes as a graph of groups $Y$ and one of the vertex groups $Y_q$ must be isomorphic to the fundamental group of a surface $\Sigma$ with boundary, equipped with a measured foliation $\mathcal{F}$ with dense leaves, for which  the action of $Y_q$ on the minimal invariant subtree $(T_\infty)_{Y_q}$ is equivariantly isomorphic to the action of $\pi_1(\Sigma)$ on the $\mathbb R$-tree dual to lifted measured foliation on the universal cover $(\tilde\Sigma,\tilde{\mathcal{F}})$. 
        The subgroup $D$, elliptic in  $T_\infty$, must also lie in some vertex group $Y_D$. Since $D$ is a free group of rank 2 that fixes a point, it cannot be conjugate into the subgroup $Y_q$, so $Y_q\neq Y_D$.

        By the definition of a graph of actions (see \cite[\S 1.3]{guirardel_actions_2008}), the edge group incident to $Y_q$ must be a point stabilizer, which in turn must either be trivial or conjugate to the $\pi_1$-image of $\partial \Sigma.$ The former case gives a free decomposition of $F_3$ relative to $D$ which contradicts earlier considerations. The latter case implies that all the edges adjacent to $q$ have cyclic edge groups. It also follows that if we collapse all edges of $Y$ that have non-cyclic edge group to get a new graph of groups $\bar Y$, $Y_q$ is still a vertex group of this collapsed splitting and $\bar Y$ has another vertex group containing $D$. It follows that $F_3$ admits a non-trivial cyclic splitting relative to $D$ containing $Y_q$ as a \emph{quadratically hanging (QH) subgroup}, that is to say it is isomorphic to $\pi_1(\Sigma)$ where $\Sigma$ is a compact surface with  boundary and the (conjugacy classes) of the incident edge groups coincide with the $\pi_1$-image of $\partial \Sigma$.
        
        Because $F_3$ is one ended relative to $D$, by \cite[Theorem 9.5]{guirardel_jsj_2017} $F_3$ admits a canonical cyclic JSJ decomposition $J$ relative to $D$. By our description of $\bar Y$ above we know that $J$ has at least two vertex groups and that one of them is a maximal QH subgroup. Arc stabilizers of the Bass-Serre tree dual to $J$ are cyclic and there are at least two vertex orbits so we can apply, for example \cite{GaboriauLevitt}, to conclude that $J$ has exactly two non-cyclic vertex groups and that these vertex groups have rank exactly 2. One of these vertex groups contains (a conjugate of) $Y_q$ and is a maximal QH subgroup. The other vertex group contains $D$. The automorphism invariance of JSJ decompositions, \cite[Corollary 7.4]{guirardel_jsj_2017}, implies that $\phi$ must preserve $J$, i.e. $\phi$ maps vertex groups and edge groups to conjugates of vertex and edge groups (respectively). Since $Q$ is the unique flexible subgroup of $J$ we must have that $\phi(Q)$ is mapped to a conjugate of $Q.$

        Now the vertex group $Y_q$ from the graph of actions $Y$ sits inside the QH subgroup $Q$ as the $\pi_1$-image of some subsurface, since both groups have the same rank, $Y_q=Q$ (up to conjugacy). By \cite{culler_vogtmann} the only Seifert-type action of $F_2$ is dual to an irrational foliation on a once punctured torus. It follows that the QH strucutre on $Q$ is that of an orientable surface of genus 1 with one boundary component. Therefore there is exactly one edge group incident to $Q$.

        Let $\bar J$ be the graph of groups obtained by contracting all edges except the unique  edge adjacent to the unique QH vertex group $Q$. Then $\bar J$ is the amalgamated product $F_3=D'*_C Q$. $D'$ must contain a conjugate of $D$ and by construction this splitting is $\phi$-invariant, so $D'=D$. Since $C$ is conjugate to the $\pi_1$-image of $\partial \Sigma$ in $\pi_1(\Sigma)=Q$ we have that $Q$ is one-ended relative to $C$. Since $F_3$ many-ended, by \cite{shenitzer_decomposition_1955,swarup_decompositions_1986,touikan_one-endedness_2015}, $D$ is forced to admit a free decomposition $D=\langle d \rangle * \langle c\rangle$ with $C\leq \langle c \rangle$. This gives\[
            F_3 = \underbrace{\langle d \rangle *(\langle c \rangle}_{D}*_C Q).
        \] Since $Q$ must be the fundamental group of  a torus $\Sigma$ with a boundary component we have that $C$, the $\pi_1$-image of $\partial\Sigma$, must be a commutator and therefore cannot be a proper power, thus $C=\langle c \rangle$ so $Q$ is a free factor of $F_3$ as required.  Its uniqueness follows from the canonicity of the JSJ decomposition. 

        By \cite[Theorem II.1]{GJLL} if the stretch factor of the homotethy $H$ is $\lambda=1$, then $T_\infty$ is simplicial. Since that $Q$ acts on $(T_\infty)_Q$ with dense orbits we have $\lambda >1$. Since up to coposition with an inner automorphism we have  $\phi(Q)=Q$ we have that $\phi$ induces an exponentially growing automorphism on $Q.$
    \end{proof}

In particular, we get the following.

\begin{cor}\label{cor;no_iff2_atrh}
 Let $\Phi \in \Out(F_3)$ be exponentially growing, not fully irreducible, and such that no power has an invariant proper free factor of rank 2. Then the mapping torus 
  $ F_3 \rtimes_\phi \mathbb Z$
 is almost toral relatively hyperbolic.
\end{cor}

This concludes all cases of Flowchart \ref{fig:THE_algo}, and thus proves Theorem \ref{thm:main}.

\bibliographystyle{alpha}

\end{document}